\documentclass[12pt]{amsart}
\usepackage{amsmath,amsfonts,amsthm,amssymb,mathrsfs,amsxtra,amscd,latexsym, xcolor}
\usepackage[hmargin=2cm,vmargin=3cm]{geometry}
\usepackage{hyperref}
\usepackage{verbatim}

\usepackage{amsmath,amsthm,amssymb,enumerate,amsfonts,mathrsfs,amsxtra,amscd,latexsym}

\newtheorem{theorem}{Theorem}[section]

\newtheorem{lemma}[theorem]{Lemma}
\newtheorem{corollary}[theorem]{Corollary}

\newtheorem*{remark}{Remark}
\newtheorem*{remarks}{Remarks}

\numberwithin{equation}{section}
\definecolor{green}{HTML}{4E9A25}

\newcommand{\sm}{\Big(\begin{smallmatrix}}
\newcommand{\esm}{\end{smallmatrix}\Big)}
\newcommand{\mat}{\Big(\begin{matrix}}
\newcommand{\emat}{\end{matrix}\Big)}

\renewcommand{\epsilon}{\varepsilon}

\def\re{\mathrm{Re}}
\def\im{\mathrm{Im}}

 \def\R{\mathbb R} \def\Z{\mathbb Z} \def\C{\mathbb C}

\newcommand{\beq}{\begin{eqnarray*}}
\newcommand{\eeq}{\end{eqnarray*}}
\newcommand{\beqn}{\begin{eqnarray}}
\newcommand{\eeqn}{\end{eqnarray}}

\newcommand{\ben}{\begin{enumerate}}
\newcommand{\een}{\end{enumerate}}

\def\={\;=\;}  \def\+{\,+\,}  

\begin{document}

\title{Jensen polynomials for the Riemann xi-function}

\author{M. J. Griffin}

\address{Department of Mathematics, 275 TMCB,  Brigham Young University, 
Provo, UT 84602}
\email{mjgriffin@math.byu.edu}

\author{K. Ono}

\address{Department of Mathematics, University of Virginia, Charlottesville, VA 22904}
\email{ko5wk@virginia.edu}

\author{L. Rolen}
\address{Department of Mathematics, Vanderbilt University, Nashville, TN 37240}
\email{larry.rolen@vanderbilt.edu}

\author{J. Thorner}
\address{Department of Mathematics, University of Illinois, Urbana, IL 61801}
\email{jesse.thorner@gmail.com}

\author{Z. Tripp}
\address{Department of Mathematics, Vanderbilt University, Nashville, TN 37240}
\email{zachary.d.tripp@vanderbilt.edu}

\author{I. Wagner}
\address{Department of Mathematics, Vanderbilt University, Nashville, TN 37240}
\email{ian.c.wagner@vanderbilt.edu}

\begin{abstract}  
We investigate $\xi(s)=\frac{1}{2}s(s-1)\pi^{-\frac{s}{2}}\Gamma(\frac{s}{2})\zeta(s)$, where $\zeta(s)$ is the Riemann zeta function.  The Riemann hypothesis (RH) asserts that if $\xi(s)=0$, then $\re(s)=\frac{1}{2}$.  P\'olya  proved that RH is equivalent to the hyperbolicity of the Jensen polynomials $J^{d,n}(X)$ constructed from certain Taylor coefficients of  $\xi(s)$.  For each $d\geq 1$, recent work proves that $J^{d,n}(X)$ is hyperbolic for sufficiently large $n$.  In this paper, we make this result effective.  Moreover, we show how the low-lying zeros of the derivatives $\xi^{(n)}(s)$ influence the hyperbolicity of $J^{d,n}(X)$.
\end{abstract}

\keywords{Riemann zeta function, Riemann hypothesis, Jensen polynomial}
\thanks{2010
 Mathematics Subject Classification: 11M26, 11M06}

\maketitle

\vspace{-1mm}
\section{Introduction and Statement of Results} \label{section1}

Let $\zeta(s)$ be the Riemann zeta function.  Define $\xi(s):=\frac{1}{2}s(s-1)\pi^{-\frac{s}{2}}\Gamma(\frac{s}{2})\zeta(s)$ and\footnote{This presentation, which is convenient for us, differs from the traditional $\sum_{j=0}^{\infty}\gamma(j)z^{2j}/(2j)!$.}
\begin{equation}
\label{psi}
	\psi(z):=\sum_{j=0}^{\infty} \frac{\gamma(j)}{j!}z^{2j}=\xi\Big(\frac{1}{2}+z\Big).
\end{equation}
It is known that $\gamma(n)> 0$ for all $n\geq 0$ \cite[Section 4.4]{Broughan}.  For $d,n\geq 0$, the degree $d$ Jensen polynomial $J^{d,n}(X)$ for the $n$-th derivative $\xi^{(n)}(s)$ is
\begin{equation}
\label{eqn:Jensen_poly}
	J^{d,n}(X):=\sum_{j=0}^{d}{d\choose j}\gamma(n+j)X^j.
\end{equation}

A polynomial with real coefficients is {\it hyperbolic} if all of its zeros are real.   Expanding on notes of Jensen, P\'olya \cite{Polya} proved that the Riemann hypothesis (RH) is equivalent to the hyperbolicity of $J^{d,n}(X)$ for all $d,n\geq 0$.  Since RH remains unproved, some research has focused on proving hyperbolicity  for all $n\geq 0$ when $d$ is small.  Csordas, Norfolk, and Varga \cite{CNV} and Dimitrov and Lucas \cite{DL} proved hyperbolicity for $n\geq 0$ and $d\leq 3$. Building on the work of Borcea and Brändén \cite{BB} and Obreschkoff \cite{Obreschkoff}, Chasse \cite{Chasse} proved hyperbolicity for $d \le 2 \times 10^{17}$  and $n\ge 0$.

Recent work \cite{GORZ} provides a complementary treatment. For all $d\geq 1$, there is a threshold $N(d)$ such that $J^{d,n}(X)$ is hyperbolic for $n\geq N(d)$.  Specifically, under the transformation \eqref{Jtilde} below, the polynomials $J^{d,n}(X)$ are closely modeled by the Hermite polynomials $H_d(\frac{X}{2})$, where
\begin{equation}
\label{eqn:Hermite}
	\sum_{d=0}^\infty H_d(X)\frac{t^d}{d!} := e^{2Xt-t^2} = 1+2X t+(4X^2-2)\frac{t^2}{2!}+(8X^3-12X)\frac{t^3}{3!}+\cdots
\end{equation}
Thus for large $n$, $J^{d,n}(X)$ inherits hyperbolicity from $H_d(\frac{X}{2})$.  See Bombieri \cite{Bombieri} for commentary.
  
Our main result, which builds on work in \cite{GORZ}, provides an effective upper bound for $N(d)$.

\begin{theorem}\label{Thm1}
There is a constant $c>0$ such that $J^{d,n}(X)$ is hyperbolic for $d\geq 1$ and $n\geq ce^{d/2}$.
\end{theorem}

For an integer $m\geq 0$, let $\mathrm{RH}_m$ to be the statement that if $\xi^{(m)}(s)=0$, then $\re(s)=\frac{1}{2}$.  It is well known that $\mathrm{RH}=\mathrm{RH}_0$ implies $\mathrm{RH}_m$ for all $m\geq 1$ \cite{Polya}.  The ideas of P\'olya lead to the conclusion that $\xi^{(m)}(s)$ satisfies $\mathrm{RH}_m$ if and only if $J^{d,n}(X)$ is hyperbolic for $d\geq 1$ and $n\geq m$.  For  $T\geq 0$, we define $\mathrm{RH}_m(T)$ to be the statement that all zeros $\rho^{(m)}$ of $\xi^{(m)}(s)$ with $|\im(\rho^{(m)})|\leq T$ satisfy $\re(\rho^{(m)})=\frac{1}{2}$.  Our second result is a relationship between $\mathrm{RH}_m(T)$ and the hyperbolicity of $J^{d,n}(X)$ for $n\geq m$.  In what follows, $\lfloor x\rfloor$ denotes the usual floor function.

\begin{theorem}\label{Thm2}
If $\mathrm{RH}_m(T)$ is true and $d\leq \lfloor T\rfloor ^2$, then
$J^{d,n}(X)$ is hyperbolic for all $n\geq m$.
\end{theorem}
This is a modest generalization of work of Chasse \cite[Theorem 1.8]{Chasse}, which Theorem \ref{Thm2} recovers when $m=0$.  We include it for the sake of completeness.  Since Platt \cite{Platt} has verified $\mathrm{RH}_0(3.06\times 10^{10})$, Theorem \ref{Thm2} implies the following corollary.

\begin{corollary}
\label{cor:chasse_improved}
If $d\leq 9.36\times 10^{20}$ and $n\geq 0$, then $J^{d,n}(X)$ is hyperbolic.
\end{corollary}~
\begin{remarks}~
\begin{enumerate}
\item One can generalize the notion of a Jensen polynomial by replacing the Taylor coefficients $\gamma(n)$ with other suitable arithmetic functions in \eqref{eqn:Jensen_poly}.  Questions of hyperbolicity for such polynomials can be of great arithmetic interest \cite{GORZ}.  While some of the ideas presented here might apply in other settings, we restrict our consideration and only present the strongest conclusions for $\xi(s)$ that our methods appear to permit.
\item Our proof quantifies the rate at which a certain transformation of $J^{d,n}(X)$ tends to $H_d(\frac{X}{2})$ as $n$ tends to infinity.    See Farmer \cite{Farmer(2020)} for an interesting interpretation of this as an instance of a uniform variant of Berry's ``cosine is a universal attractor'' principle \cite{Berry}.
\item It would be most desirable to prove a sort of converse to Theorem \ref{Thm2} wherein the partial results on hyperbolicity from Theorem \ref{Thm1} would directly influence the distribution of zeros of the derivatives of $\xi(s)$, or perhaps even $\xi(s)$ itself.  While Theorem \ref{Thm2} indicates that a partial understanding of the zeros of $\xi^{(m)}(s)$ influence the hyperbolicity of $J^{d,n}(X)$ for $n\geq m$, a quick inspection of the proofs in \cite{Polya} indicates that it is highly unlikely that converse influence exists unless one has hyperbolicity for all $n\geq m$ and all $d\geq 1$.  While Jensen polynomials can be used to uniformly approximate $\xi^{(n)}(\frac{1}{2}+it)$, they are ultimately quite inefficient at detecting zeros that violate $\mathrm{RH}_n$ (should any such zeros exist).  One can see this by directly plotting the aforementioned uniform approximation.  
\item After this paper was written, O'Sullivan \cite{OSullivan} wrote an interesting paper on the P{\'o}lya-Jensen criterion for the Riemann Hypothesis. Instead of working directly with the Jensen polynomials $J^{d,n}(X)$, he considers a variant of the original criterion which makes use of $\sum_{j=0}^d \binom{d}{j} \gamma(n+j) H_{d-j}(X)$. His paper complements the explicit results obtained here for this modified criterion.
\end{enumerate}
\end{remarks}

In Section \ref{section2}, we prove Theorem \ref{Thm1} using a small modification of a result of Tur\'an.  Our proof assumes two technical results (Theorems \ref{AsymNewThm} and \ref{Precision}) that we prove in Sections \ref{thm2.1proof} and  \ref{sec:proof_Precision}.   In Section \ref{sec:proof_Chasse}, we prove Theorem \ref{Thm2}.

\subsection*{Acknowledgements}

We thank the referee for a thorough reading and helpful comments.  The second author thanks the support of the Thomas Jefferson Fund and the NSF.  The fourth author began this work while partially supported by a NSF Postdoctoral Fellowship.

\section{Proof of Theorem \ref{Thm1}}\label{section2}

The effective refinement of the work in \cite{GORZ} provided by Theorem~\ref{Thm1} uses different methods. Our proofs are facilitated by renormalizations of several objects in \cite{GORZ}.

\subsection{New conventions and preliminaries}

Recall the setup in \cite[Section 5]{GORZ}.  It was shown that for each $d\geq 1$, there exist positive numbers $A(n)$, $\delta(n)$, $ g_3(n)$, $g_4(n),\ldots,g_d(n)$ such that 
\begin{equation}\label{AsymOld}
\log\Big(\frac{\gamma(n+j)}{\gamma(n)}\Big)=A(n)j -\delta(n)^2 j^2+\sum_{i=3}^{d}g_i(n) j^i+o(\delta(n)^d),
\end{equation}
with $g_i(n)=O(n^{1-i})=o(\delta(n)^{i})$ and  $	\delta(n)\sim \frac{1}{\sqrt{2n}}$.  From these, we define
\begin{equation}
\label{Jtilde}
	\widehat{J}^{d,n}(X):= \frac{\delta(n)^{-d}}{\gamma(n)}J^{d,n} \Big(\frac{\delta(n) X-1}{\exp(A(n))}\Big).
\end{equation}
Estimates in \cite{GORZ} are written in terms of the behavior of $\delta(n)$, and there is considerable latitude in the choice of $\delta(n)$.  In this sense, $\delta(n)$ serves as a uniformizer for the calculations in \cite{GORZ}.

We introduce a more refined uniformizer
\begin{equation}\label{delta}
\Delta(M):= \sqrt{\frac{1}{2}\Big(1-\frac{\gamma(M-2)\gamma(M)}{\gamma(M-1)^2}\Big)}
\end{equation}
and the a normalization $\widetilde{J}^{d,n}(X)$ of the polynomials $J^{d,n}(X)$.  It will become apparent that $\Delta(M)$ is a more convenient and more accurate uniformizer than $\delta(n)$, which is important for our eventual goal of an effective lower bound for $n$ in terms of $d$.  Before defining $\widetilde{J}^{d,n}(X)$, we establish some basic properties of $\Delta(M)$.  As a consequence of the hyperbolicity of $J^{2,n}(X)$ \cite{CV}, we know that $\gamma(n-2)\gamma(n)\leq \gamma(n-1)^2$ for all $n\geq 3$.  This establishes the log concavity of $\gamma(n)$.  It follows that $\Delta(M)\in\R$ for all $M\geq 3$.  The next theorem contains some key results for $\Delta(M)$.

\begin{theorem}\label{AsymNewThm}
Let $\Delta(M)$ be as in \eqref{delta}.
\begin{enumerate}
	\item We have $\Delta(M)\sim 1/\sqrt{2M}$.  In particular, if $C>1$, then there exists $M_C>C/(C-1)$ (depending only on $C$) such that if $M>M_C$, then $1/\sqrt{2C(M-1)}\leq\Delta(M)\leq 1/\sqrt{M}$.
\item For each integer $m\geq 1$, there exists a function $G_m(z)$, holomorphic for $\re(z)>1$, such that for all integers $1\leq j< M$ we have 
\begin{equation}\label{AsymNew}
\log\Big(\frac{\gamma(M-j)}{\gamma(M)}\Big)=-\sum_{m=1}^\infty G_m(M)\Delta(M)^{2m-2} j^m.
\end{equation}
With $C>1$ as in part (1), the bound $|G_m(M)|\ll_C(2C)^{m}$ holds for all integers $m,M\geq 1$.  We also have the limit $\lim_{M\to \infty}G_m(M)=\frac{2^{m-1}}{m(m-1)}$.
\item We have 
\begin{equation}\label{G2}
G_2(M)=1+(1-3G_3(M))\Delta(M)^2+O(\Delta(M)^{4}).
\end{equation}
\end{enumerate}
\end{theorem}
We will prove Theorem \ref{AsymNewThm} in Section \ref{thm2.1proof}.
\begin{remark}
The uniform bound on $|G_m(M)|$ is critical for our proofs. While $G_m(M)$ is a bounded function of $M$ for {\it fixed} $m$, we need to bound $|G_m(M)|$ when $m$ and $M$ vary jointly.
\end{remark}

Now that we have listed some key properties of $\Delta(M)$, we define
\begin{equation}\label{Jhat}
\widetilde{J}^{d,n}(X):=
 \frac{{\color{black}\gamma(n+d)^{d-1}}}{{\color{black}\gamma(n+d-1)^{d}}\cdot \Delta(n+d)^{d}} J^{d,n}\Big(\frac{\gamma(n+d-1)}{\gamma(n+d)}\cdot (\Delta(n+d)X-1)\Big).
 \end{equation}
For future convenience, we define the coefficients $A_{d,k}(n)$ by the expansion
\begin{equation}
\label{eqn:Jdn_Adk}
\widetilde{J}^{d,n}(X)=\sum_{k=0}^d A_{d,k}(n)X^{d-k}.
\end{equation}
The following lemma explains our reason for working with these new normalizations.

\begin{lemma}
\label{3term}
If $d\geq 1$ and $n \ge 0$, then $A_{d,0}(n)=1$, $A_{d,1}(n)=0$, and $A_{d,2}(n)=-d(d-1)$.  In particular, $\widetilde{J}^{1,n}(X)=H_1(\frac{X}{2})$, $\widetilde{J}^{2,n}(X)=H_2(\frac{X}{2})$, and $\deg\big(\widetilde{J}^{d,n}(X)-H_d(\frac{X}{2})\big)\leq d-3$ for $d\geq 3$.
\end{lemma}
\begin{proof}
This is straightforward to verify from \eqref{eqn:Jensen_poly}, \eqref{eqn:Hermite}, and (\ref{Jhat}).
\end{proof}

We use Theorem \ref{AsymNewThm} to prove asymptotics for the coefficients $A_{d,k}(n)$ for $k\geq 3$.

\begin{theorem}
\label{Precision}
Let $d\geq 4$, $n\geq 0$, and $3 \le k \le d$ be integers, and let $1<C<2$.  Recall the definition of $M_C$ from Theorem \ref{AsymNewThm}(1).  If
$n+d>\max\{10k^3,M_C\}$, then
{\small
\[
\frac{(-1)^{\lfloor \frac{k}{2}\rfloor}(d-k)!\lfloor \frac{k}{2}\rfloor!}{d!}A_{d,k}(n)=\begin{cases}
	1+Z_{n+d}(\lfloor \frac{k}{2}\rfloor) \Delta(n+d)^2+O_C(k^{6}(4C)^k\Delta(n+d)^4)&\mbox{if $k$ is even,}\\
\lfloor \frac{k}{2}\rfloor (G_3(n+d)-2)\Delta(n+d)+O_C(k^4 (4C)^k \Delta(n+d)^3)&\mbox{if $k$ is odd,}
\end{cases}
\]}
where $Z_{n+d}(t):=t(t-1) ( -\frac{2}{3}(3t+2) +2t G_3(n+d) - \frac{t - 2}{2}G_3(n+d)^2 - G_4(n+d))$.
\end{theorem}

We prove Theorem \ref{Precision} in Section \ref{sec:proof_Precision}.

\subsection{Proof of Theorem \ref{Thm1}}

We use the following result to prove Theorem \ref{Thm1}.

\begin{lemma}
\label{lem:turan}
For $0\leq j\leq d$, define
\begin{equation}
\label{c_d,n,j}
c_{d,n,j} := \sum_{i=0}^{\lfloor \frac{j}{2}\rfloor} \frac{(d-j+2i)!}{i!(d-j)!}A_{d,j-2i}(n),
\end{equation}
where $A_{d,k}(n)$ is defined by \eqref{eqn:Jdn_Adk}.  If
\begin{equation}\label{explicit Turan criterion}
\sum_{j=3}^d 2^{-j} \frac{(d-j)!}{(d-1)!} c_{d,n,j}^2 < 1,
\end{equation}
then $J^{d,n}(X)$ is hyperbolic.
\end{lemma}
\begin{proof}
There exist $A,B,C\in\R$ (depending on $n$ and $d$) such that $\widetilde{J}^{d,n}(X)=AJ^{d,n}(BX+C)$, hence $J^{d,n}(X)$ is hyperbolic if and only $\widetilde{J}^{d,n}(X)$ is hyperbolic.  We apply the inversion formula \cite[Equation 18.18.20]{DLMF} to \eqref{eqn:Jdn_Adk} and obtain $\widetilde{J}^{d,n}(X) = \sum_{j=0}^d c_{d,n,j}H_{d-j}(\frac{X}{2})$.  Tur{\'a}n \cite[Theorem III]{Turan} proved that if $c_j\in\R$ for $0\leq j\leq N$ and
	\begin{equation}
	\label{eqn:turan_bound}
	\sum_{j=0}^{N-2} 2^j j! c_j^2 < 2^N (N-1)! c_N^2,
	\end{equation}
	then all roots of $\sum_{j = 0}^N c_j H_j(z)$ (hence $\sum_{j = 0}^N c_j H_j(\frac{z}{2})$) are real and simple.  Since $c_{d,n,0} = 1$ and $c_{d,n,1} = c_{d,n,2} = 0$ by Lemma \ref{3term}, the inequality \eqref{eqn:turan_bound} applied to our setting reduces to \eqref{explicit Turan criterion}.
\end{proof}

\begin{proof}[Proof of Theorem \ref{Thm1}]
We will show that there exists a suitably large absolute constant $c>0$ such that if $n\geq ce^{d/2}$, then \eqref{explicit Turan criterion} holds, in which case Lemma \ref{lem:turan} applies.  We now appeal to Theorem \ref{Precision}.  When $j = 2\ell$, we use the even case of Theorem \ref{Precision}, \eqref{c_d,n,j}, and the fact that $A_{d,0} = 1$ and $A_{d,2} = -d(d-1)$ to find that $c_{d,n,2\ell}$ equals
\begin{multline*}\label{asym c_d,n,j}
 \sum_{i=0}^{\ell} \frac{(d-2i)!}{(\ell - i)!(d-2\ell)!}A_{d,2i}(n) \nonumber \\
=\frac{d!}{\ell!(d-2\ell)!} \Big[ \sum_{i=0}^\ell {\ell \choose i } (-1)^i + \Delta(n+d)^2\sum_{i=2}^\ell {\ell \choose i } (-1)^i Z_{n+d}(i) + O_C\Big(\Delta(n+d)^4\sum_{i=2}^\ell {\ell \choose i } i^6 (4C)^{2i} \Big) \Big]\\
=\frac{d!}{\ell!(d-2\ell)!} \Big[ \Delta(n+d)^2\sum_{i=2}^\ell {\ell \choose i } (-1)^i Z_{n+d}(i) + O_C\Big(\Delta(n+d)^4\sum_{i=2}^\ell {\ell \choose i } i^6 (4C)^{2i} \Big) \Big].
\end{multline*}

For a function $f$ defined on the nonnegative integers, we define the $k$-th difference operator
\begin{equation}
\label{eqn:differencing_operator}
\sigma_{k,x}(f(x)) := \sum_{j=0}^k (-1)^{k-j}{k\choose j}f(j)	
\end{equation}
Note that $f(x)$ is given by polynomial of degree at most $d$ if and only if  $\sigma_{k,x}(f(x))=0$ for all $k>d$. Since $Z_{n+d}(t)$ is a polynomial in $t$ of degree $3$ with $Z_{n+d}(0) = Z_{n+d}(1) = 0$, it follows if $\ell\geq 4$, then $\sum_{i=2}^\ell {\ell \choose i } (-1)^i Z_{n+d}(i)=0$.  Thus if $\ell \ge 4$, then we apply the bound $i^6\leq \ell^6$ to conclude that
\begin{equation}
\label{even c asymptotics}
	c_{d,n,2\ell} \ll_C \frac{d!}{(d-2\ell)!\ell!}\Delta(n+d)^4  \ell^6\sum_{i=2}^\ell \begin{pmatrix}\ell \\ i \end{pmatrix} (4C)^{2i} \ll_C \frac{d!}{(d-2\ell)!\ell!}\ell^6 (16C^2 + 1)^\ell \Delta(n+d)^4.
\end{equation}
The bound \eqref{even c asymptotics} also holds when $\ell = 2$ and $\ell = 3$ by bounding the main terms directly. A symmetric calculation using the odd case of Theorem \ref{Precision} reveals that
\begin{equation}\label{odd c asymptotics}
	c_{d,n,2\ell + 1} \ll_C \frac{d!}{(d-2\ell - 1)! \ell! }\ell^4 (16C^2 + 1)^\ell \Delta(n+d)^3.
\end{equation}

The bound \eqref{even c asymptotics} leads to a bound for the even-indexed terms in \eqref{explicit Turan criterion}, namely
\begin{equation*}
\sum_{\substack{3\leq j\leq d \\ \textup{$j$ even}}}2^{-j}\frac{(d-j)!}{(d-1)!}c_{d,n,j}^2\ll_C d\Delta(n+d)^8 \sum_{1\leq \ell\leq d/2} \binom{d}{2\ell} \binom{2\ell}{\ell} \ell^{12} \Big(\frac{16C^2 + 1}{4}\Big)^\ell.
\end{equation*}
Note that $\binom{2\ell}{\ell} \sim \frac{4^\ell}{\sqrt{\pi \ell}}$ by Stirling's formula.  Trivially bounding $\ell^{12}\leq d^{12}$, we find that
\begin{equation}
\label{eqn:even_piece}
	\sum_{\substack{3\leq j\leq d \\ \textup{$j$ even}}}2^{-j}\frac{(d-j)!}{(d-1)!}c_{d,n,j}^2\ll_C d^{13}(1+\sqrt{1+16C^2})^{d}\Delta(n+d)^8.
\end{equation}
A similar bound over the odd terms holds as well:
\begin{equation}
\label{eqn:odd_piece}
\sum_{\substack{3\leq j\leq d \\ \textup{$j$ odd}}}2^{-j}\frac{(d-j)!}{(d-1)!}c_{d,n,j}^2\ll_C d^{9}(1+\sqrt{1+16C^2})^d \Delta(n+d)^6.
\end{equation}
We combine \eqref{eqn:even_piece} and \eqref{eqn:odd_piece} with the bound for $\Delta(n+d)$ in Theorem \ref{AsymNewThm} to conclude that there is a constant $\alpha_C>0$ such that \eqref{explicit Turan criterion} holds if $n\geq \alpha_C d^{13/3}(1+\sqrt{1+16C^2})^{d/3}$.  Per Corollary \ref{cor:chasse_improved}, we may assume that $d\geq 9.36\cdot 10^{20}$.  
We choose  $C=1+10^{-5}$, in which case there exists a constant $c>0$ such that $\alpha_C d^{13/3}(1+\sqrt{1+16C^2})^{d/3}\leq ce^{d/2}$, as desired.
\end{proof}

\section{Proof of Theorem \ref{AsymNewThm}}
\label{thm2.1proof}

Define
 \begin{equation*}
 \label{FnDefn}
 F(z) :=  \int_1^\infty (\log t)^z\,t^{-3/4}\,\Big(\sum_{k=1}^\infty e^{-\pi k^2 t}\Big)\,dt,
 \end{equation*}
which is holomorphic for $\re(z)>0$.  It follows from \cite[Equation 13]{GORZ} that
\begin{equation}\label{gamma}
\gamma(M)
=\frac{M!}{(2M)!} \cdot\,\frac{32 \binom {2M}2F(2M-2) - F(2M)}{2^{2M-1}}.
\end{equation}
If we replace the binomial coefficient and factorials in (\ref{gamma}) with $\Gamma$-functions, we see that (\ref{gamma}) extends to a function of a {\it complex} variable $M$ which is holomorphic for $\re(M)>1$.

For $M>0$, let  $L_M$ be the unique positive solution of the equation $M=L_M(\pi e^{L_M}+\frac{3}{4})$.  It is straightforward to show that $L_M\sim \log(\frac{M}{\log M})$.  Define $K_M=(L_M^{-1}+L_M^{-2})M-\frac{3}{4}$.  The function $L_M$ (and therefore $K_M$) extends to a function which is holomorphic and non-vanishing for $\re(M)>1$.  By \cite[Equation 16]{GORZ}, we have
{\small\begin{equation}\label{GammaExpression}
\gamma(M)=
\frac{e^{M-2}M^{M+\frac{1}{2}} L_{2M-2}^{2M-2}}{2^{2M-5}{(2M-2)}^{(2M-2)+\frac{1}{2}}}
 \sqrt{\frac{2\pi}{K_{2M-2}}}
\exp\Big(\frac{L_{2M-2}}{4} -\frac{2M-2}{L_{2M-2}}+\frac{3}{4} \Big) \Big(1+O_{\varepsilon}\Big(\frac{1}{M^{1-\varepsilon}}\Big)\Big).
\end{equation}}
Ultimately, the analytic continuation of $L_M$ and Stirling's formula imply that even when $M$ is complex, we may keep the existing error term in \eqref{GammaExpression} once we replace $M$ with $|M|$.

For fixed $\re(M)>1$, there is a function $R_M(j)$ of a {\it complex} variable $j$, holomorphic and non-vanishing for $|j|<\re(M)-1$, with the property that if $j,M\in\Z$ satisfy $|j|<M$, then
\begin{equation}\label{Rgamma}
R_M(j)=\gamma(M-j)/\gamma(M).
\end{equation}
Since $R_M(j)$ is holomorphic and nonvanishing when $|j|<\re(M)-1$, we have the expansion
\begin{equation}\label{AsymTemp}
\log R_M(j)=\sum_{m=1}^\infty a_m(M) j^m,\qquad |j|<\re(M)-1.
\end{equation}
By varying $M$, we find the Taylor coefficents $a_m(M)$ are in fact holomorphic functions in $M$.

Since $\log R_M(j)$ is holomorphic for $M$ and $j$ in the specified domains, the right hand side of \eqref{AsymTemp} converges absolutely and uniformly for $j$ and $M$ in \emph{compact} subsets of their respective domains. We wish to give bounds on the coefficients $a_m(M)$ which are uniform for \emph{all} real $M$ and $j$ in their respective domains. To do so, we must regularize $\log R_M(j)$ to obtain a function $R^*_{M}(\lambda)$ which extends to a function of $M$ on the extended interval $[3,\infty]$. For convenience, we replace $j$ with $\lambda(M-2)$; it suffices to consider $\lambda$ in the closed disk $|\lambda|\leq 1$ (rather than $j$ in a domain that varies with $M$).  We now define our regularized function
\begin{equation}
\label{RMod}
R_{M}^*(\lambda):=\frac{1}{M-2}\log\Big(\Big(\frac{eL_{2M-2}^2 M}{4(2M-2)^2}\Big)^{\lambda(M-2)}R_M(\lambda(M-2))\Big) +(1-\lambda)\log(1-\lambda).
\end{equation}
Our expansion for $R_{M}^*(\lambda)$ for $|\lambda|\leq 1$ naturally incorporates the coefficients $a_m(M)$:
\begin{equation}
\label{RModAsym}
R_{M}^*(\lambda) = \Big(a_1(M) +\log\Big(\frac{eL_{2M-2}^2 M}{4(2M-2)^2}\Big)-1\Big)\lambda  + \sum_{m=2}^\infty \Big(a_m(M) (M-2)^{m-1} +\frac{1}{m(m-1)}\Big)\lambda^m.
\end{equation}
\begin{lemma}
\label{lem:R_star}
The function $R_{M}^*(\lambda)$ is holomorphic for all $|\lambda|\leq 1$ and all $M$ in the extended interval $[3,\infty]$.  Moreover, for all $|\lambda|\leq 1$, we have that  $\lim_{M\to \infty}R_{M}^*(\lambda)=0$
\end{lemma}

\begin{proof}

For all finite $M$ and $|\lambda|\leq 1$, the function $R_{M}^*(\lambda)$ is holomorphic since each such point corresponds to a value of $R_M(j)$ with $|j|\leq M-2$. In order to understand the behavior of $R_{M}^*(\lambda)$ as $M\to\infty$, we consider the regularized limit
\begin{equation}\label{RLim}
\lim_{M\to \infty}\frac{1}{M-2}\log\Big(\Big(\frac{eL_{2M-2}^2M}{4(2M-2)^2}\Big)^{\lambda(M-2)}R(\lambda(M-2);M)\Big).
\end{equation}
Let $j=\lambda(M-2)$, as above.  The asymptotic \eqref{GammaExpression} implies that as $M\to\infty$, we have
\[
\frac{1}{M-2}\log\Big(\Big(\frac{eL_{2M-2}^2 M}{4(2M-2)^2}\Big)^{j}R_M(j)\Big)= 
A+B+C+O_{\varepsilon}\Big(\frac{\log M}{M}\Big),
\]
where
\begin{align*}
	A_{M}(\lambda)&= \frac1{M-2}\log\Big( \frac{M^j(M-j)^{M-j-2}(2M-2)^{2M-2+\frac12}}{(2M-2)^{2j}({2M-2-2j})^{{2M-2-2j}+\frac{1}{2}}M^{M-2}}\Big),
\\
B_{M}(\lambda)&=\frac{1}{M-2}\Big((2M-2-2j)\log\Big(\frac{{L_{{2M-2-2j}}}}{L_{2M-2}}\Big) -\frac{1}{2}\log\Big(\frac{{K_{{2M-2-2j}}}}{K_{2M-2}}\Big)\Big),\quad\textup{and}
\\
C_{M}(\lambda)&=\frac{1}{M-2}\Big( \frac{{L_{{2M-2-2j}}}}{4} -\frac{{2M-2-2j}}{{L_{{2M-2-2j}}}}-\frac{L_{2M-2}}{4} +\frac{2M-2}{L_{2M-2}}\Big).
\end{align*}
Since $L_M\sim\log(\frac{M}{\log M})$, a calculus exercise shows that $\lim_{M\to\infty}B_M(\lambda)=\lim_{M\to\infty}C_{M}(\lambda)=0$.

Simplifying $A_M(\lambda)$, we find that
\[
A_M(\lambda)=\frac{M-j-2}{M-2}\log \Big(1-\frac{j}{M}\Big)-\frac{{2M-2-2j}+\tfrac12}{M-2}\log\Big(1-\frac{2j}{2M-2}\Big).
\]
Since $j=\lambda(M-2)$, it follows that
\begin{equation}
\label{eqn:michael_limit_above}
\lim_{M\to \infty}A_M(\lambda)=-(1-\lambda)\log(1-\lambda)=
\lambda-\sum_{m=2}^{\infty}\frac{1}{m(m-1)}\lambda^m.
\end{equation}
The rightmost sum converges absolutely for $|\lambda|< 1$, but is not holomorphic at $1$, hence we remove the term in \eqref{RMod} so that \eqref{RModAsym} converges on the boundary of the disk.
\end{proof}

Since $R_{M}^*(\lambda)$ is holomorphic for $|\lambda|\leq 1$ and all $M\in [3,\infty]$, the Taylor series given in \eqref{RModAsym} converges absolutely and uniformly for all such $\lambda$ and $M$. Taking $\lambda=1$ and $M\geq 3$, we find that for all $\epsilon>0$, there exists an integer $W_{\epsilon}\geq1$, {\it depending only on $\epsilon$}, such that
\begin{equation}
\label{amAsym}
|a_m(M)(M-2)^{m-1}+(m(m-1))^{-1}|<\epsilon\qquad\textup{whenever $m\geq W_{\epsilon}$.}
\end{equation}

\begin{proof}[Proof of Theorem \ref{AsymNewThm}]

To prove our claimed asymptotic for $\Delta(M)$, we write $\Delta(M)$ in terms of $a_m(M)$.  We extend $\Delta(M)$, originally defined in \eqref{delta}, to a holomorphic function by the identity
\begin{equation}
\label{Delta_to_R}
\Delta(M)= \sqrt{\frac{1}{2}\Big(1-\frac{R_M(2)}{R_M(1)^2}\Big)}.
\end{equation}
We use \eqref{AsymTemp} to expand \eqref{Delta_to_R} and then apply \eqref{amAsym} to bound $a_m(M)$ for $m\geq 4$, thus obtaining
\begin{equation}\label{DeltaExp}
\Delta(M)= \sqrt{-a_2(M)-3a_3(M)-a_2(M)^2+O(M^{-3})}.
\end{equation}
The asymptotic $\Delta(M)\sim \frac{1}{\sqrt{2M}}$ now follows from (\ref{amAsym}) as we let $\epsilon\to 0$.

We define $G_m(M)$ by the identity $-a_m(M)=G_m(M)\Delta(M)^{2m-2}$.  The expansion \eqref{AsymNew} now has the desired properties, and the claimed bounds and asymptotics for $G_m(M)$ follow from \eqref{amAsym} and the fact that $\Delta(M)\sim\frac{1}{\sqrt{2M}}$.  To recover (\ref{G2}), we square both sides of (\ref{DeltaExp}), and notice that $G_2(M)$ satisfies the quadratic equation 
 \[
 \Delta(M)^2 G_2(M)^2-G_2(M)+1-3G_3(M)\Delta(M)^2=O(\Delta(M)^{4}).
 \]
 The desired result follows.
\end{proof}

\begin{remark}
These methods provide an effective alternative to the approach to asymptotics for $g_m(n)$ and $\delta(n)$ in \cite{GORZ}.  Greater care is required here than in \cite{GORZ} because of the uniformity required in Theorem \ref{Thm1}.  Comparing (\ref{AsymOld}) and (\ref{AsymNew}), and noticing the sign change of $j$ on the left hand side both equations, we see that $G_m(M)\Delta(M)^{2m-2}\sim (-1)^{m+1}g_m(M)$.  In particular, we see that $g_2(M)\sim -1/(2M)$, which implies that $\delta(M)\sim \Delta(M)$.
\end{remark}

\section{Proof of Theorem \ref{Precision}}
\label{sec:proof_Precision}

Using the functions $G_m(M)$ given by Theorem \ref{AsymNewThm}, we define $S(j;M)$ and $Q_m(M)$ as follows:
\begin{align}\label{S_exp}
S(j;M)=\frac{R_M(j)}{R_M(1)^j}=\exp\Big(\sum_{m=2}^\infty G_m(M)\Delta(M)^{2m-2} (j-j^m)\Big)=\sum_{m=0}^\infty Q_m(M)j^m.
\end{align}
This definition of $S(j;M)$ is critical because, by \eqref{Rgamma}, we have for integers $j\in[0,M-1]$ that
\[
S(j;M)=\frac{\gamma(M-j)\gamma(M)^{j-1}}{\gamma(M-1)^j}.
\]
Using (\ref{Jhat}), we may rewrite the coefficients $A_{d,k}(n)$ as
\begin{equation}\label{A1}
A_{d,k}(n)= {d\choose k} \Delta(n+d)^{-k}\sum_{j=0}^k (-1)^{k-j} {k\choose j} S(j;n+d).
\end{equation}
Recall \eqref{eqn:differencing_operator}, and define $y_{m,k}=\sigma_{k,x}(x^m)$.  This leads to the identity
\begin{equation}\label{A2}
A_{d,k}(n)= {d\choose k} \Delta(n+d)^{-k}\sum_{m=0}^\infty y_{m,k}Q_m(n+d).
\end{equation}
We have the following lemma about the size of the $y_{m,k}$.

\begin{lemma}\label{smk}
Let $y_{m,k}$ be defined as above. Then $y_{m,k}=0$ if $m<k$, and
\[
y_{k,k}=k!, \quad y_{k+1,k}=k!{k+1\choose 2}, \quad y_{k+2,k}=k!{k+2\choose 3}\frac{3k+1}{4},\quad y_{k+3,k}=k!{k+3\choose 4}\frac{k^2+k}{2}.
\]
More generally, for all $i\geq 1$, there exists a polynomial $P_{i}(k)$ of degree $i-1$, satisfying $P_i(1)=1$ and $P_i(k)\leq k^{i-1}$ for all positive integers $k$, such that $y_{k+i,k}=k!{k+i\choose 1+i} P_{i}(k)$.
\end{lemma}
\begin{proof}
If $m<k$, the identity $y_{m,k}=0$ follows the discussion following \eqref{eqn:differencing_operator}.  For $m\geq k$, we have the identity $(e^X-1)^k = ( X + \frac{X^2}{2} + \frac{X^3}{3!} + \dots )^k = \sum_{m=0}^\infty \frac{y_{m,k}}{m!}X^m$.  For integers $t>i$, we now consider $\sigma_{t,k}(\frac{y_{k+i,k}}{(k+i)!})$ as a function of $k$. For fixed $t$, we have the generating function 
\begin{eqnarray*}
\sum_{i=0}^{\infty} \sigma_{t,k}\Big(\frac{y_{k+i,k}}{(k+i)!}\Big)X^{i+t} 
&=& \sum_{i=0}^{\infty} \sum_{j=0}^{t}(-1)^{t-j}{t\choose j}\Big(\frac{y_{j+i,j}}{(j+i)!}\Big)X^{i+t} \\
&=& \sum_{j=0}^{t} (-1)^{t-j}{t\choose j}X^{t-j} \sum_{i=0}^{\infty}\Big(\frac{y_{j+i,j}}{(j+i)!}\Big) X^{j+i}\\
&=& \sum_{j=0}^{t} (-1)^{t-j}{t\choose j}X^{t-j} (e^X-1)^j= (e^X-X-1)^t = \frac{1}{2^{t}}X^{2t}+\cdots.
\end{eqnarray*}
Hence $\sigma_{t,k}(\frac{y_{k+i,k}}{(k+i)!})=0$ for $t>i$.  This implies that $\frac{y_{k+i,k}}{(k+i)!}$ is a polynomial in $k$ of degree at most $i$.  For $i\geq 1$, note that $y_{i,0}=0$, and $y_{i,1}=1$.  Thus, we can factor $y_{k+i,k}$ as
\[
y_{k+i,k}=k P_{i}(k)\prod_{j=0}^{i-1}\frac{k+i-j}{1+i-j}=k!{k+i\choose i+1} P_{i}(k),
\]
where $P_{i}(1)=1$. A short calculation gives the claimed expressions for $y_{k+1,k}$, $y_{k+2,k}$, and $y_{k+3,k}$.

We prove that $P_i(k)\leq k^{i-1}$ for all positive integers $k$ by comparing the Taylor coefficients of 
\begin{equation}
\label{eqn:expansions_k}
\Big(\frac{e^{X}-1}{X}\Big)^k=\sum_{i=0}^\infty\frac{k\cdot P_i(k)}{(i+1)!}X^{i}\qquad\text{and}\qquad 
\frac{e^{kX}-1}{kX}=\sum_{i=0}^{\infty}\frac{k^{i}}{(i+1)!}X^{i}.
\end{equation}
Given functions $f=f(x)$ and $g=g(x)$ which are analytic at $0$, let $f \prec g$ denote the condition that $f^{(i)}(0)\leq g^{(i)}(0)$ for all integers $i\geq 0$.  In other words, the $i$-th Taylor coefficient of $g$ is at least the $i$-th Taylor coefficient of $f$ in the expansions at zero.   This statement has transitivity---if $f\prec g$ and $g\prec h$, then $f\prec h$.  If $h^{(i)}(0)\geq 0$ for all $i\geq 0$, then $f\prec g$ implies $fh\prec gh$.

By comparing the expansions in \eqref{eqn:expansions_k}, the bound $P_i(k)\leq k^{i-1}$ is equivalent to
\begin{equation}
\label{eqn:step_1}
\Big(\frac{e^x-1}{x}\Big)^k\prec \frac{e^{kx}-1}{kx}.
\end{equation}
Define $F_k=F_k(x):=(e^{kx/2}-e^{-kx/2})/(kx)$.  We rewrite \eqref{eqn:step_1} as $e^{kx/2}F_1^k\prec e^{kx/2}F_k$.  Since $(e^{kx/2})^{(n)}(0)>0$ for all $n\geq 0$, $e^{kx/2}F_1^k\prec e^{kx/2}F_k$ follows from $F_1^k\prec F_k$.

We will prove $F_1^k\prec F_k$ by induction on $k$.  The result when $k=1$ is trivial.  Suppose now that $F_1^k\prec F_k$ is true for some integer $k\geq 1$.  By transitivity, the truth of $F_1^{k+1}\prec F_{k+1}$ follows from that of $F_1^{k+1}\prec F_1 F_k$ and $F_1 F_k\prec F_{k+1}$.  Since $F_k^{(i)}(0)\geq 0$ for all $i\geq 0$, our inductive hypothesis $F_1^k\prec F_k$ implies $F_1^{k+1}\prec F_1 F_k$.

It remains to prove $F_1 F_k\prec F_{k+1}$ for all $k\geq 1$.  We have the expansions
{\small\[
F_1(x) F_k(x)=\sum_{i=1}^{\infty}\frac{2}{k(2i)!}\Big(\Big(\frac{k+1}{2}\Big)^{2i}-\Big(\frac{k-1}{2}\Big)^{2i}\Big)x^{2i-2},\quad F_{k+1}(x)=\sum_{i=1}^{\infty}\frac{1}{(2i-1)!}\Big(\frac{k+1}{2}\Big)^{2i-2}x^{2i-2}.
\]}
Thus $F_1 F_k\prec F_{k+1}$, and hence \eqref{eqn:step_1}, follows from the bound
\begin{equation}
\label{eqn:step_3}
0\leq \frac{4ik}{(k+1)^{2}}+\Big(\frac{k-1}{k+1}\Big)^{2i}-1,\qquad i,k\geq 1.
\end{equation}
Denote the right side of \eqref{eqn:step_3} as $\omega_k(i)$.  Observe that $\omega_k(i+1)-\omega_k(i) = \frac{4k}{(k+1)^2}(1-(\frac{k-1}{k+1})^{2i})>0$  for all $i,k\geq 1$.  It follows that $\omega_k(i)\geq \omega_k(1)=0$ for all $i,k\geq 1$, which proves \eqref{eqn:step_3}.
\end{proof}

The desired asymptotic for $A_{d,k}(n)$ will follow from a suitable bound for $Q_m(M)$, which we prove using Theorem \ref{AsymNewThm} and Lemma \ref{smk}.

\begin{lemma}\label{bound Q_m}
If $m,\ell\geq 1$ are integers, $C>1$, and $M \ge \max(\ell^3, M_C)$, then $|Q_m(M)|\ell! \ll_C (4C)^m \ell^{\ell - \frac{1}{2}m}\Delta(M)^m$.
\end{lemma}
\begin{proof}
Let $\lambda$ be a partition of $m$, denoted $\lambda\vdash m$.  Let $\lambda_i$ be the number of parts equal to $i$ so that $\sum_{i=1}^m i\lambda_i=m$.  Define $\mathcal{L}(\lambda)=\sum_{i=1}^m\lambda_i$.  From (\ref{S_exp}) and the multinomial theorem, we obtain
\begin{align}\label{partition sum}
\frac{Q_m(M)\ell!}{\Delta(M)^m} &= \frac{\ell!}{\Delta(M)^m} \sum_{\lambda \vdash m} \frac{(\widetilde{G}_1(M) \Delta(M)^2)^{\lambda_1}}{\lambda_1!} \frac{(-G_2\Delta(M)^2)^{\lambda_2}}{\lambda_2!}\cdots \frac{(-G_m\Delta(M)^{2m-2})^{\lambda_m}}{\lambda_m!}\nonumber\\
&= \sum_{\lambda\vdash m} (-1)^{\mathcal{L}(\lambda)-\lambda_1} \frac{\ell!}{\lambda_1!\lambda_2!\cdots \lambda_m!} \widetilde{G}_1(M)^{\lambda_1}G_2(M)^{\lambda_2}\cdots G_m(M)^{\lambda_m} \Delta(M)^{m-2\mathcal{L}(\lambda)+2\lambda_1},
\end{align}
where $\widetilde{G}_1(M):=\sum_{m=2}^{\infty} G_{m}(M) \Delta(M)^{2m-4}$.  Since $|G_i(M)|\ll_C(2C)^i$ by Theorem \ref{AsymNewThm}, it follows that $\widetilde{G}_1(M) = 1+O_C(\Delta(M)^2)$ and $|\widetilde{G}_1(M)^{\lambda_1}G_2(M)^{\lambda_2}\cdots G_m(M)^{\lambda_m}|\ll_C (2C)^{\sum_{i=1}^m i\lambda_i}=(2C)^{m}$.  Since $\Delta(M) \le M^{-\frac{1}{2}}\le \ell^{-\frac{3}{2}}$ by Theorem \ref{AsymNewThm} and our hypotheses, the definition of $\mathcal{L}(\lambda)$ yields
\[
\frac{\ell!}{\lambda_1!\lambda_2!\cdots \lambda_m!} \Delta(M)^{m-2\mathcal{L}(\lambda)+2\lambda_1} \le \frac{\ell!}{\lambda_2!} \ell^{-\frac{3}{2}(m-2\mathcal{L}(\lambda)+2\lambda_1)} \le \ell^{\ell - \lambda_2 - \frac{3}{2}(m-2\mathcal{L}(\lambda)+2\lambda_1)}\leq \ell^{\ell - \frac{1}{2}m}.
\]
The desired result now follows since there are at most $2^m$ partitions of $m$.
\end{proof}

\begin{proof}[Proof of Theorem \ref{Precision}]
Recall \eqref{A2}, which expresses $A_{d,k}(n)$ as a sum of $y_{m,k}Q_m(n+d)$ over $m\geq 0$.  We use Lemma \ref{smk} to rewrite the contribution from $y_{m,k}$ in \eqref{A2} and arrive at
\begin{multline*}
A_{d,k}(n)={d\choose k}k!\Big[\frac{Q_{k}(n+d)}{\Delta(n+d)^k}+{k+1\choose 2}\frac{Q_{k+1}(n+d)}{\Delta(n+d)^k}+{k+2\choose 3}\frac{3k+1}{4}\frac{Q_{k+2}(n+d)}{\Delta(n+d)^k}\\
+{k+3\choose 4}\frac{k^2+k}{2}\frac{Q_{k+3}(n+d)}{\Delta(n+d)^k}+\sum_{i=4}^{\infty}{k+i\choose 1+i}\frac{P_i(k)Q_{k+i}(n+d)}{\Delta(n+d)^k}\Big].
\end{multline*}
Let $j=\lfloor k/2\rfloor$.  Since ${d\choose k}k!=\frac{d!}{(d-k)!}$, it follows that $A_{d,k}(n)$ equals
\begin{multline*}
	\frac{(-1)^{j}d!}{j!(d-k)!}\Big[\frac{(-1)^{j}j! Q_{k}(n+d)}{\Delta(n+d)^k}+{k+1\choose 2}\frac{(-1)^{j}j!Q_{k+1}(n+d)}{\Delta(n+d)^k}\\
	+{k+2\choose 3}\frac{3k+1}{4}\frac{(-1)^{j}j!Q_{k+2}(n+d)}{\Delta(n+d)^k}+{k+3\choose 4}\frac{k^2+k}{2}\frac{(-1)^{j}j!Q_{k+3}(n+d)}{\Delta(n+d)^k}\\
	+\sum_{i=4}^{\infty}{k+i\choose 1+i}P_i(k)\frac{(-1)^{j}j!Q_{k+i}(n+d)}{\Delta(n+d)^k}\Big].
\end{multline*}
Suppose that $n+d>\max\{j^3,M_C,64C^2j\}$.  The asymptotic bounds for $\Delta(n+d)$ from Theorem \ref{AsymNewThm} and the bound for $Q_{k+i}(n+d)$ in Lemma \ref{bound Q_m} imply that $A_{d,k}(n)$ equals
\begin{equation}
\label{eqn:main_suM_22}
	\begin{aligned}
		\frac{(-1)^{j}d!}{j!(d-k)!}&\Big[\frac{(-1)^{j}j! Q_{k}(n+d)}{\Delta(n+d)^k}+{k+1\choose 2}\frac{(-1)^{j}j!Q_{k+1}(n+d)}{\Delta(n+d)^k} \\
		&+{k+2\choose 3}\frac{3k+1}{4}\frac{(-1)^{j}j!Q_{k+2}(n+d)}{\Delta(n+d)^k}+{k+3\choose 4}\frac{k^2+k}{2}\frac{(-1)^{j}j!Q_{k+3}(n+d)}{\Delta(n+d)^k}\\
		&\hspace{9cm}+O((4C)^{k}k^{9/2}\Delta(n+d)^4)\Big].
	\end{aligned}
\end{equation}

Let $m\in\{k,k+1,k+2,k+3\}$.  As in Lemma \ref{bound Q_m}, we use \eqref{partition sum} to expand $Q_{m}(n+d)$, bounding the contribution from the partitions $\lambda$ such that $m-2\mathcal{L}(\lambda)+2\lambda_1\geq 3$ using the bound for $|G_m(n+d)|$ in Theorem \ref{AsymNewThm}.  Since $m-2\mathcal{L}(\lambda)+2\lambda_1 = \lambda_1+\sum_{i=3}^{m}(i-2)\lambda_i$, we must separately consider the cases where $m$ is even (where the powers of $\Delta(n+d)$ are even) and $m$ is odd (where the powers of $\Delta(n+d)$ are odd).  When $M=n+d$ and $m$ is even, it then follows from \eqref{partition sum} that $Q_{m}(M)$ equals $(-1)^{m/2} \Delta(M)^m / (\frac{m}{2})!$ times
\begin{equation}
\label{eqn:m_even}
	\begin{aligned}
		&G_2(M)^{\frac{m}{2}} - \frac{m}{4}\Big( G_2(M)^{\frac{m}{2}-1} \widetilde{G}_1(M)^2 + (m -2) G_4(M) G_2(M)^{\frac{m}{2} - 2} \\
		&+(m-2) G_3(M) G_2(M)^{\frac{m}{2} - 2}\widetilde{G}_1(M)+ \frac{(m-2)(m-4)}{4}G_3(M)^2G_2(M)^{\frac{m}{2} - 3}\Big)\Delta(M)^2\\
		&+O_C(m^6 (4C)^m \Delta(M)^{4}).	
	\end{aligned}
\end{equation}
Similarly, when $m$ is odd, $Q_m(M)$ equals $(-1)^{\lfloor\frac{m}{2}\rfloor} \Delta(M)^m/ (\lfloor\frac{m}{2}\rfloor)!$ times
\begin{equation}\label{eqn:m_odd}
(G_2(M)^{\lfloor\frac{m}{2}\rfloor} \tilde{G}_1(M) + \lfloor\tfrac{m}{2}\rfloor G_3(M) G_2(M)^{\lfloor\frac{m}{2}\rfloor - 1})\Delta(M) + O_C(m^4 (4C)^m \Delta(M)^3).
\end{equation}
The theorem follows by substituting \eqref{eqn:m_even} and \eqref{eqn:m_odd} into \eqref{eqn:main_suM_22}.
\end{proof}

\section{Proof of Theorem \ref{Thm2}}
\label{sec:proof_Chasse}

We introduce some notation.  For $0<\delta< \pi/2$, define $S(\theta,\delta):=\{z\in\mathbb{C}^{\times}\colon |\arg(z)-\theta|\leq \delta\}$.  Let $C(\theta,\delta)$ to be the set of entire functions $F$ such that there exist a sequence of complex numbers $(\beta_k)_{k\geq 1}$, an integer $q\geq 0$, and constants $c,\sigma\in\mathbb{C}$ such that $ \sum_{k=1}^{\infty}\frac{1}{|\beta_k|}<\infty$, $\beta_k,\sigma\in S(\theta,\delta)$, and
 \[
 F(z)= cz^{q}e^{-\sigma z}\prod_{k=1}^{\infty}\Big(1-\frac{z}{\beta_k}\Big).
 \]

\begin{lemma}\label{Chasse1}
Let $0<\delta< \pi/2$.  If $F\in C(\theta,\delta)$, then $F$ is locally uniformly approximated by polynomials, each of whose zeros lie in $S(\theta,\delta)$, and conversely.  Moreover, if $m\geq 1$ is an integer and the $m$-th derivative $F^{(m)}$ is not identically zero, then $F^{(m)}\in C(\theta,\delta)$.
\end{lemma}
\begin{proof}
The first claim is proved in \cite[Chapter VIII]{Levin}. For the second claim, suppose that $F\in C(\theta,\delta)$ is non-constant.  By the first claim, there exists a sequence of nonzero polynomials $(g_n)$ which locally uniformly approximate $F$, and each zero of $g_n$ lies in $S(\theta,\delta)$.  By the Gauss-Lucas theorem, the zeros of $g_n'$ belong to the convex hull of the set of zeros of $g_n$; thus each zero of $g_n'$ lies in $S(\theta,\delta)$.  Since the sequence $(g_n')$ locally uniformly approximates $F'$, it follows by the first claim that $F'\in C(\theta,\delta)$.  For higher derivatives, we proceed by induction.
\end{proof}
\begin{lemma}
\label{Chasse2}
	If $\frac{d^n}{dz^n}\psi(\sqrt{z})\in C(\pi,\delta)$, then $J^{d,n}(X)$ is hyperbolic for $d\leq|\sin(\delta)|^{-2}$.
\end{lemma}
\begin{proof}
	In \eqref{psi}, all powers of $z$ are even, so $\psi(\sqrt{z})$ is entire. Since $\gamma(j)>0$ for all $j\geq 0$ and
	\[
	\frac{d^n}{dz^n}\psi(\sqrt{z}) = \sum_{j=0}^{\infty}\frac{\gamma(j+n)}{j!}z^j,
	\]
	the Taylor coefficients of $\frac{d^n}{dz^n}\psi(\sqrt{z})$ are positive.  Hence the lemma follows immediately from \cite[Theorem 3.6]{Chasse} with $\varphi=\frac{d^n}{dz^n}\psi(\sqrt{z})$.
\end{proof}

\begin{proof}[Proof of Theorem \ref{Thm2}]
We follow \cite{Chasse}.  Let $m\geq 0$ be an integer.  Suppose that $\mathrm{RH}_m(T)$ holds for some $T>\frac{1}{2}$.  Then the zeros of $\frac{d^m}{dz^m}\psi(z)$ in the rectangle $\{z\in\C\colon |\re(z)|< 1/2,~|\im(z)|\leq T\}$ are imaginary.  Therefore, the zeros of $\frac{d^m}{dz^m}\psi(\sqrt{z})$ must lie in $S(0,2\arctan(\frac{1}{2T}))\cup S(\pi,2\arctan(\frac{1}{2T}))$.  Since $\gamma(j)>0$ for all $j\geq 0$, the zeros of $\frac{d^m}{dz^m}\psi(\sqrt{z})$ lie in the half-plane $\re(z)<0$, and hence must lie in $S(\pi,2\arctan(\frac{1}{2T}))$.  Hence $\frac{d^m}{dz^m}\psi(\sqrt{z})\in C(\pi,2\arctan(\frac{1}{2T}))$.  We see from Lemma \ref{Chasse2} that $J^{d,m}(X)$ is hyperbolic for $d\leq \lfloor |\sin(2\arctan(\tfrac{1}{2T}))|^{-2}\rfloor = \lfloor T^2+\tfrac{1}{2}+\tfrac{1}{16T^2}\rfloor$.  Thus if $d\leq \lfloor T\rfloor^2$, then $J^{d,m}(X)$ is hyperbolic.  Since $C(\theta,\delta)$ is closed under differentiation per Lemma \ref{Chasse1}, we have $\frac{d^n}{dz^n}\psi(\sqrt{z})\in C(\pi,2\arctan(\frac{1}{2T}))$ for all $n\geq m$.  This finishes the proof.
\end{proof}

\end{document}